\documentclass{amsart}
\usepackage[margin=2cm]{geometry}
\usepackage{graphics, amsmath,amssymb, enumitem, comment, url, mathtools}
\usepackage{graphicx}
\usepackage{epsfig}

\newif\ifcolorcomments
\newcommand{\allowcomments}[4]{
\newcommand{#1}[1]{\ifdraft{\ifcolorcomments{\textcolor{#4}{##1 --#3}}\else{\textsl{ ##1 \ --#3}}\fi}\else{}\fi}
}

\usepackage[dvipsnames]{xcolor}

\colorcommentstrue
\usepackage{amssymb}
\usepackage{amsmath,amsthm}

\usepackage{colortbl}

\newtheorem{theorem}{Theorem}[section]
\newtheorem{lemma}[theorem]{Lemma}
\newtheorem{proposition}[theorem]{Proposition}
\newtheorem{corollary}[theorem]{Corollary}
\theoremstyle{definition}

\newtheorem{remark}[theorem]{Remark}

\newcommand{\D}{\mathbb D}

\newcommand{\EE}{\mathcal E}

\newcommand{\FF}{\mathcal F}

\newcommand{\HH}{\mathcal H}

\newcommand{\N}{\mathbb N}

\newcommand{\R}{\mathbb R}

\newcommand{\UU}{\mathcal U}

\newcommand{\mbf}{\mathbf}

\newcommand{\dd}{\mbf d}

\renewcommand{\text}{\textup}

\newcommand{\NPC}[1]{\ignorespaces}

\renewcommand{\H}{\mathbb H}
\renewcommand{\SS}{\mathcal S}

\newif\ifdraft\drafttrue

\def\N{\mathbb N}

\def\R{\mathbb R}

\def\D{\mathcal D}
\def\H{\mathcal H}

\newcommand\hdim{\dim_{\mathrm H}}
\newcommand{\subscript}[2]{$#1 _ #2$}

\IfFileExists{marvosym.sty}{
\RequirePackage{marvosym} 
}{

}

\IfFileExists{wasysym.sty}{
\RequirePackage{wasysym}\renewcommand{\emptyset}{{\diameter}}
}{

}

\newcommand*{\myDots}{\ifmmode\mathellipsis\else.\kern-0.07em.\kern-0.07em.\fi}
\DeclarePairedDelimiter\floor{\lfloor}{\rfloor}

\newcommand*{\defeq}{\stackrel{\text{def}}{=}}

\allowcomments{\commumtaz}{MH}{Mumtaz}{green}
\allowcomments{\comnikita}{NS}{Nikita}{blue}
\allowcomments{\combixuan}{BL}{Bixuan}{red}

\newcommand {\ignore}[1] {}

\newcommand{\commenty}[1]{}

\begin{document}

\title[Hausdorff dimension for exponentially growing partial quotients]{Metrical properties of exponentially growing partial quotients}

\author[Mumtaz Hussain]{Mumtaz Hussain}
\address{Mumtaz Hussain,  Department of Mathematical and Physical Sciences,  La Trobe University, Bendigo 3552, Australia. }
\email{m.hussain@latrobe.edu.au}

\author{Nikita Shulga}
\address{Nikita Shulga,  Department of Mathematical and Physical Sciences,  La Trobe University, Bendigo 3552, Australia. }
\email{n.shulga@latrobe.edu.au}
\date{}

\maketitle

\numberwithin{equation}{section}
\begin{abstract}
A fundamental challenge within the metric theory of continued fractions involves quantifying sets of real numbers, when represented using continued fractions, exhibit partial quotients that grow at specific rates. For any positive function $\Phi$, Wang-Wu theorem (2008) comprehensively describes the Hausdorff dimension of the set
\begin{equation*}
\EE_1(\Phi):=\left\{x\in [0, 1): a_n(x)\geq \Phi(n) \ {\rm for \ infinitely \ many} \ n\in \N\right\}. 
\end{equation*}
Various generalisations of this set exist, such as substituting one partial quotient with the product of consecutive partial quotients in the aforementioned set which has connections with the improvements to Dirichlet's theorem, and many other sets of similar nature. Establishing the upper bound of the Hausdorff dimension of such sets is significantly easier than proving the lower bound.  In this paper, we present a unified approach to get an optimal lower bound for many known setups, including results by Wang-Wu [Adv. Math.,
2008],  Huang-Wu-Xu  [Israel J. Math. 2020], Bakhtawar-Bos-Hussain [Nonlinearity 2020], and several others, and also provide a new theorem derived as an application of our main result. We do this by finding an exact Hausdorff dimension of the set 
$$S_m(A_0,\ldots,A_{m-1}) \defeq \left\{ x\in[0,1): \,  c_i A_i^n \le a_{n+i}(x) < 2c_i A_i^n,0 \le i \le m-1 \  \text{for infinitely many } n\in\N \right\},$$
where each partial quotient grows exponentially and the base is given by a parameter $A_i>1$. For proper choices of $A_i$'s, this set serves as a subset for sets under consideration, providing an optimal lower bound of Hausdorff dimension in all of them. The crux of the proof lies in introducing of multiple probability measures consistently distributed over the Cantor-type subset of $S_m(A_0,\ldots,A_{m-1})$. 
\end{abstract}

\section{Introduction}

It is well-known that every irrational number $x\in (0, 1)$ has a unique infinite continued fraction expansion. This expansion can be induced by the Gauss map  $T: [0,1)\to [0,1)$ defined by
\[T(0)=0, ~ T(x)=\frac{1}{x}-\floor*{\frac{1}{x}} \textmd{ for }x\in(0,1),\]
where $\lfloor x\rfloor$ denotes the integer part of $x$.  We write $x:=[a_{1}(x),a_{2}(x),a_{3}(x),\ldots ]$ for the continued fraction of $x$, where $a_1(x)=\lfloor 1/x \rfloor$, $a_{n}(x)= a_1(T^{n-1}(x))$ for $n\ge2$ are called the partial quotients of $x$ (or continued fraction digits of $x$). 
%
The metric theory of continued fractions concerns the quantitative study of properties of partial quotients for almost all $x\in(0, 1)$. This area of research is closely connected with metric Diophantine approximation, for example, fundamental theorems in this field by Khintchine (1924) and Jarn\'ik (1931) can be represented in terms of the growth of partial quotients. The classical Borel-Bernstein theorem (1912) states that the Lebesgue measure of the set 
\begin{equation*}
\EE_1(\Phi):=\left\{x\in [0, 1): a_n(x)\geq \Phi(n) \ {\rm for \ infinitely \ many} \ n\in \N\right\}
\end{equation*}
is either zero or one depending upon the convergence or divergence of the series $\sum_{n=1}^\infty \Phi(n)^{-1}$ respectively. Here and throughout  $\Phi:\N\to [1, \infty)$ will be an arbitrary function such that $\Phi(n)\to\infty$ as $n\to \infty$.  For rapidly increasing functions $\Phi$, the Borel-Bernstein (1911, 1912) theorem gives no  further information than the zero Lebesgue measure of $\EE_1(\Phi)$. To distinguish between the sets of zero Lebesgue measure,  the Hausdorff dimension is an appropriate tool. 
 For an arbitrary function $\Phi$, the dimension of $\EE_1(\Phi)$ was computed by Wang-Wu \cite{WaWu08}.

The consideration of the growth of product of consecutive partial quotients played a significant role in understanding the uniform approximation theory (improvements to Dirichlet's theorem as opposed to improvements to Dirichlet's corollary). In particular, the set
\begin{equation*}
\EE_m(\Phi):=\left\{x\in [0, 1): \prod_{i=0}^{m-1}a_{n+i}(x)\geq \Phi(n) \ {\rm for \ infinitely \ many} \ n\in \N\right\}
\end{equation*}
received significant attention recently. See  \cite{KleinbockWadleigh}  for the Lebesgue measure of $\EE_2(\Phi)$, \cite{BHS, HKWW} for the Hausdorff measure of $\EE_2(\Phi(q_n))$,  \cite{HuWuXu} for the Lebesgue measure and Hausdorff dimension of $\EE_m(\Phi)$, and  \cite{BBH2, HLS} for the Hausdorff dimension of difference of sets $\EE_2(\Phi)\setminus \EE_1(\Phi)$. 

In this paper, we provide a unified treatment to all of these results (and some others) and retrieve them from our main theorem proved below. Given the breadth of the generality we envisage there will be many more applications other than those listed below.

\subsection{Main result} 
For a fixed integer number $m$ and for all integers $0\le i\le m-1$,  let $A_i>1$ be a real number. Define the set
$$S_m=S_m(A_0,\ldots,A_{m-1}) := \left\{ x\in[0,1): \,  c_i A_i^n \le a_{n+i}(x) < 2c_i A_i^n,\, 0 \le i \le m-1\quad \text{for infinitely many } n\in\N \right\},$$
where $c_i\in \R_{>0}$.  For any $0 \le i \le m-1$, define the quantities
$$
\beta_{-1} = 1,\,\,\, \beta_i = A_0\cdots A_{i} .
$$
Let
\begin{equation}\label{defd}
d_i =\inf\{ s\ge0 : P( T, -s \log | T'(x)| - s \log \beta_i + (1-s) \log \beta_{i-1})\le 0\},
\end{equation}
{where  $P(T, \psi)$ is a pressure function, the definition is given in Section \ref{Pressure Functions}. }

The main result of the paper is the following theorem.
\begin{theorem}\label{General}
$$\hdim S_m = \min_{0 \le i \le m-1} d_i.$$
\end{theorem}
{
 \begin{remark}
 Even though the set $S_m$ depends on $c_0,\ldots,c_{m-1}$, the exact values of these constants does not  change the Hausdorff dimension of the set.
 \end{remark}

\subsection{Applications (summary)}As applications of our theorem, we obtain optimal lower bounds of Hausdorff dimension of the sets listed below, the detailed proofs are given in section 4. Note that, at their own, the proofs of the lower bound of the Hausdorff dimension of these sets were the main ingredients in the papers listed next to them, excluding the last set which is a new application.

\begin{itemize}
\item Wang-Wu \cite{WaWu08}.
$$
F(B) := \{ x\in [0,1) : a_n(x) \ge B^n  \ \text{\,for infinitely many } n\in\N \}.
$$
\item Bakhtawar-Bos-Hussain \cite{BBH2}.
\begin{equation*}\label{pqc}
 \mathcal{F}(B ):=\EE_2(B)\setminus\EE_1(B)=\left\{ x\in [
0,1):
\begin{array}{r}
a_{n+1}(x)a_{n}(x)\geq B^n \text{ for infinitely many }n\in 
\mathbb{N}
\text{ and} \\ 
a_{n+1}(x)<B^n \text{ for all sufficiently large }n\in 
\mathbb{N}
\end{array}
\right\}.
\end{equation*}

\item Hussain-Li-Shulga \cite{HLS}.
$$
E(A_1,A_2): = \left\{ x\in[0,1): \, c_1 A_1^n\leq a_n(x) <2c_1A_1^n,  \, \,  c_2 A_2^n\leq a_{n+1}(x)<2  c_2 A_2^n,\,\text{for infinitely many }n\in\N \right\}
$$
and
\begin{equation*}
\FF_{B_1,B_2}:=\left\{x\in[0,1):
\begin{split} 
 a_n(x)a_{n+1}(x) & \geq B_1^n \text{\,\, for infinitely many } n\in\N \\
 a_{n+1}(x) & < B_2^n  \text{\,\, for all sufficiently large } n\in\N
\end{split}
\right\}.
\end{equation*}

\item Huang-Wu-Xu \cite{HuWuXu}. For $m\geq1$, the set
$$
E_m (B) := \{ x\in[0,1) : a_n(x) a_{n+1}(x) \cdots a_{n+m-1} (x) \ge B^n \  \text{for infinitely many } n\in\N \}
$$

\item Bakhtawar-Hussain-Kleinbock-Wang \cite{BHKW}.
\begin{equation*}\label{weight}
\D_2^{\mathbf t}(B):=\left\{x\in[0, 1):  a_{n}^{t_0} a_{n+1}^{t_1} \ge B^n \ {\text{for infinitely many}} \ n\in \N \right\}.
\end{equation*}

\item Tan-Tian-Wang  \cite{TanTianWang}.
$$
E(\psi): = \{ x\in [0,1) : \exists 1 \le k\neq l \le n , a_k(x) \ge \psi(n), a_l(x) \ge \psi(n), \ \text{for infinitely many } n\in\N \}.
$$

\item 
For $m\ge2$, the set
\begin{equation}
\FF^m_{B_1,B_2}:=\left\{x\in[0,1):
\begin{split} 
 a_{n}(x)\cdots a_{n+m-1}(x) & \geq B_1^n \text{\,\, for infinitely many } n\in\N \\
 a_{n+1}(x)\cdots a_{n+m-1}(x) & < B_2^n  \text{\,\, for all sufficiently large } n\in\N
\end{split}
\right\},
\end{equation} 
which is a new application of Theorem \ref{General}. We also provide an upper bound of Hausdorff dimension for this set.
\end{itemize}
}


\medskip
The proof of Theorem \ref{General} will be given in Section 3. In Section 4 we apply this theorem to the sets listed above, showing that Theorem \ref{General} gives an optimal lower bound in all of them. 

\noindent {\bf Acknowledgements.}   This research is supported by the Australian Research Council Discovery Project (200100994).  We thank Bixuan Li for useful discussions.

\section{Preliminaries and auxiliary results}\label{S2}

For completeness we give a brief introduction to Hausdorff measures and dimension.  For further details we refer to \cite{BernikDodson, Falconer_book}.
\subsection{Hausdorff measure and dimension}\label{HM}\
Let $s\geq0$ and 
$E\subset \R^n$.
 Then, for any $\rho>0$ a countable collection $\{B_i\}$ of balls in
$\R^n$ with diameters $\mathrm{diam} (B_i)\le \rho$ such that
$E\subset \bigcup_i B_i$ is called a $\rho$-cover of $E$.
Let
\[
\H_\rho^s(E)=\inf \sum_i \mathrm{diam}(B_i)^s,
\]
where the infimum is taken over all possible $\rho$-covers $\{B_i\}$ of $E$. It is easy to see that $\H_\rho^s(E)$ increases as $\rho$ decreases and so approaches a limit as $\rho \rightarrow 0$. This limit could be zero or infinity, or take a finite positive value. Accordingly, the \textit{$s$-Hausdorff measure $\H^s$} of $E$ is defined to be
\[
\H^s(E)=\lim_{\rho\to 0}\H_\rho^s(E).
\]
It is easy to verify that Hausdorff measure is monotonic and countably sub-additive, and that $\H^s(\varnothing)=0$. Thus it is an outer measure on $\R^n$.
For any subset $E$ one can verify that there exists a unique critical value of $s$ at which $\H^s(E)$ `jumps' from infinity to zero. The value taken by $s$ at this discontinuity is referred to as the \textit{Hausdorff dimension of  $E$} and  is denoted by $\hdim E $; i.e.,
\[
\hdim E :=\inf\{s\in \R_+\;:\; \H^s(E)=0\}.
\] When $s=n$,  $\H^n$ coincides with standard Lebesgue measure on $\R^n$.
Computing Hausdorff dimension of a set is typically accomplished in two steps: obtaining the upper and lower bounds separately.
Upper bounds often can be handled by finding appropriate coverings. When dealing with a limsup set, one 
 {usually applies} the Hausdorff measure version of the famous Borel--Cantelli lemma (see Lemma 3.10 of \cite{BernikDodson}):
\begin{proposition}\label{bclem}
    Let $\{B_i\}_{i\ge 1}$ be a sequence of measurable  sets in $\R$ and suppose that,  $$\sum_i \mathrm{diam}(B_i)^s \, < \, \infty.$$ Then  $$\H^s({\limsup_{i\to\infty}B_i})=0.$$
\end{proposition}
The main tool in establishing the lower bound for the dimension of $ S_m(A_0,\ldots,A_{m-1})$ will be the following well-known mass distribution principle \cite{Falconer_book}.
\begin{proposition}[\cite{Falconer_book}, Mass Distribution Principle]\label{p1}
Let $\mu$ be a probability measure supported on a measurable set $F$. Suppose there are positive constants $c$ and $r_0$ such that
$$\mu(B(x,r))\le c r^s$$
for any ball $B(x,r)$ with radius $r\le r_0$ and center $x\in F$. Then $\hdim F\ge s$.
\end{proposition}
\subsection{Continued fractions and Diophantine approximation}
 Recall that the Gauss map $T: [0,1)\to [0,1)$ is defined by
\[T(0)=0, ~ T(x)=\frac{1}{x}-\floor*{\frac{1}{x}} \textmd{ for }x\in(0,1),\]
where $\lfloor x\rfloor$ denotes the integer part of $x$. 
We write $x:=[a_{1}(x),a_{2}(x),a_{3}(x),\ldots ]$ for the continued fraction of $x$ where $a_1(x)=\lfloor 1/x \rfloor$, $a_{n}(x)= a_1(T^{n-1}(x))$ for $n\ge2$ are called the partial quotients of $x$. 
The sequences $p_n= p_n(x)$, $q_n= q_n(x)$, referred to as $n^{\rm th}$ convergents, has the recursive relation
\begin {equation}\label{recu}
p_{n+1}=a_{n+1}(x)p_n+p_{n-1}, \ \
q_{n+1}=a_{n+1}(x)q_n+q_{n-1},\ \  n\geq 0.
\end {equation}
Thus $p_n=p_n(x), q_n=q_n(x)$ are determined by the partial quotients $a_1,\dots,a_n$, so we may write $p_n=p_n(a_1,\dots, a_n), q_n=q_n(a_1,\dots,a_n)$. When it is clear which partial quotients are involved, we denote them by $p_n, q_n$ for simplicity.
For any integer vector $(a_1,\dots,a_n)\in \N^n$ with $n\geq 1$, write
\begin{equation*}\label{cyl}
I_n(a_1,\dots,a_n):=\left\{x\in [0, 1): a_1(x)=a_1, \dots, a_n(x)=a_n\right\}
\end{equation*}
for the corresponding `cylinder of order $n$', that is, the set of all real numbers in $[0,1)$ whose continued fraction expansions begin with $(a_1, \dots, a_n).$
We will frequently use the following well known properties of continued fraction expansions.  They are explained in the standard texts \cite{IosKra_book, Khi_63}.
\begin{proposition}\label{pp3} For any {positive} integers $a_1,\dots,a_n$, let $p_n=p_n(a_1,\dots,a_n)$ and $q_n=q_n(a_1,\dots,a_n)$ be defined recursively by \eqref{recu}. {Then:}
\begin{enumerate}[label={\rm (\subscript{\rm P}{\arabic*})}]
\item One has
\begin{eqnarray*}
I_n(a_1,a_2,\dots,a_n)= \left\{
\begin{array}{ll}
         \left[\frac{p_n}{q_n}, \frac{p_n+p_{n-1}}{q_n+q_{n-1}}\right)     & {\rm if }\ \
         n\ {\rm{is\ even}};\\
         \left(\frac{p_n+p_{n-1}}{q_n+q_{n-1}}, \frac{p_n}{q_n}\right]     & {\rm if }\ \
         n\ {\rm{is\ odd}}.
\end{array}
        \right.
\end{eqnarray*}
{\rm Its length is given by}
\begin{equation*}\label{lencyl}
\frac{1}{2q_n^2}\leq |I_n(a_1,\ldots,a_n)|=\frac{1}{q_n(q_n+q_{n-1})}\leq \frac1{q_n^2}.
\end{equation*}

\item For any $n\geq 1$, $q_n\geq 2^{(n-1)/2}$ and
$$
1\le \frac{q_{n+m}(a_1,\ldots,a_n, b_1,\ldots, b_m)}{q_n(a_1,\ldots, a_n)\cdot q_m(b_1,\ldots,b_m)}\le 2.
$$
\item $$\prod_{i=1}^na_i\leq q_n\leq \prod_{i=1}^n(a_i+1)\leq 2^n\prod_{i=1}^na_i.$$
\item \begin{equation*}\label{p7}
\frac{1}{3a_{n+1}(x)q^2_n(x)}\, <\, \Big|x-\frac{p_n(x)}{q_n(x)}\Big|=\frac{1}{q_n(x)(q_{n+1}(x)+T^{n+1}(x)  q_n(x))}\, < \,\frac{1}{a_{n+1}q^2_n(x)},
\end{equation*}
and for any $n\geq1$, the derivative of $T^n$ is given by
$$
(T^n)'(x) = \frac{  (-1)^n}{(x q_{n-1}-p_{n-1})^2}.
$$
\item There exists a constant $K>1$ such that for almost all $x\in [0,1)$, $$
q_n(x)\le K^n, \ {\text{for all $n$ sufficiently large}}.
$$
\end{enumerate}
\end{proposition}
Let $\mu$ be the Gauss measure given by $$
d\mu=\frac{1}{(1+x)\log 2}dx.
$$ It is clear that $\mu$ is $T$-invariant and equivalent to Lebesgue measure $\mathcal{L}$.
The next proposition concerns the position of a cylinder in $[0,1)$.
\begin{proposition}[{\cite[Khintchine]{Khi_63}}]\label{pp2} Let $I_n=I_n(a_1,\dots, a_n)$ be a cylinder of order $n$, which is partitioned into sub-cylinders $\{I_{n+1}(a_1,\dots,a_n, a_{n+1}): a_{n+1}\in \N\}$. When $n$ is odd, these sub-cylinders are positioned from left to right, as $a_{n+1}$ increases from 1 to $\infty$; when $n$ is even, they are positioned from right to left.
\end{proposition}
The following result is due to {\L}uczak \cite{Luczak}.
\begin{lemma}[{\cite[ {\L}uczak]{Luczak}}]\label{lemb}For any $b, c>1$, the sets
\begin{align*}
&\left\{x\in[0, 1):  a_{n}(x)\ge c^{b^n}\  {\text{for infinitely many}} \ n\in \N       \right\},\\
&\left\{x\in[0, 1):  a_{n}(x)\ge c^{b^n}\  {\text{for all }} \ n\geq 1 \right\},
\end{align*}
have the same Hausdorff dimension $\frac1{b+1}$.
\end{lemma}

\subsection{Pressure function}\label{Pressure Functions}

When dealing with the Hausdorff dimension problems in non-linear dynamical systems, pressure functions and other concepts from thermodynamics are good tools. The concept of a general pressure function was introduced by Ruelle in \cite{Ruelle} as a generalisation of entropy which describes the exponential growth rate of ergodic sums.   We are interested in a way of obtaining the Hausdorff dimension of certain sets using the pressure functions.
A method in \cite[Theorem 2.2.1]{Pollicott2005} can be used to calculate the Hausdorff dimension of self-similar sets for linear system. As for the non-linear setting, the relation between Hausdorff dimension and pressure functions is given in \cite{Pollicott2005} as the corresponding generalisation of Moran \cite{Moran}. 
More details and context on pressure functions can be found in \cite{MaUr96,MaUr99,MaUr03}. We use the fact that the pressure function with a continuous potential can be approximated by the pressure function restricted to the sub-systems in continued fractions.

Let $\mathcal{A}\subset\mathbb{N}$ be a finite or infinite set
and define 
\begin{equation*}
X_{\mathcal{A}}=\{x\in [0,1):a_{n}(x)\in 
\mathcal{A}\,\, {\text{for all}}\ n\geq 1 \}.
\end{equation*}
Then $(X_{\mathcal{A}},T)$ is a subsystem of $([0,1),T)$ where $T$ is a
Gauss map. Given any real function $\psi:[0,1)\rightarrow \mathbb{R},$ the pressure function restricted to
the system $(X_{\mathcal{A}},T)$ is defined by 
\begin{equation}\label{5}
\mathsf{P}_{\mathcal{A}}(T,\psi):=\lim_{n\rightarrow \infty }\frac{1}{n}\log \sum_{a_{1},\ldots ,a_{n}\in \mathcal{A}}\sup_{x\in X_{\mathcal{A}}}e^{S_{n}\psi([a_{1},\ldots ,a_{n}+x])},  
\end{equation}
where $S_{n}\psi (x)$ denotes the ergodic sum $\psi(x)+\cdots
+\psi (T^{n-1}x)$. For simplicity, we denote $\mathsf{P}_{\mathbb{N}}(T,\psi)$ by $
\mathsf{P}(T,\psi )$ when $\mathcal{A}=\mathbb{N}$. We note that the supremum in equation \eqref{5} can be removed if $\psi$ satisfy the continuity property.
For each $n\geq1$, the $n^{\rm th}$ variation of $\psi$ is denoted by 
\begin{equation*}
{\mathop{\rm{Var}}}_{n}(\psi):=\sup\Big\{|\psi(x)-\psi(y)|:
I_n(x)=I_n(y)\Big\},
\end{equation*}
where $I_n$ is defined in \eqref{cyl}.\\
The following result \cite[Proposition 2.4]{LiWaWuXu14} ensure the existence of the limit in the equation (\ref{5}).
\begin{proposition}[{\protect\cite[Li-Wang-Wu-Xu]{LiWaWuXu14}}]\label{prop26}
\label{pp1} Let $\psi:[0,1)\to \mathbb{R}$ be a real function with $
\mathrm{Var}_1(\psi)<\infty$ and $\mathrm{Var}_{n}(\psi)\to 0$ as $
n\to \infty$. Then the limit defining $\mathsf{P}_\mathcal{A}(T,\psi)$
exists and the value of $\mathsf{P}_\mathcal{A}(T,\psi)$ remains the same
even without taking supremum over $x\in X_\mathcal{A}$ in \eqref{5}.
\end{proposition}
The system $([0,1),T)$ is approximated by its subsystems $(X_{\mathcal{A}},T)
$ then the pressure function has a continuity property in the system of
continued fractions. A detailed proof can be seen in \cite[Proposition 2]{HaMaUr} or \cite{LiWaWuXu14}.
\begin{proposition}[{\protect\cite[Hanus-Mauldin-Urba\'nski]{HaMaUr}}]
\label{l5} Let $\psi:[0,1)\to \mathbb{R}$ be a real function with $
\mathrm{Var}_1(\psi)<\infty$ and $\mathrm{Var}_{n}(\psi)\to 0$ as $n\to
\infty$. We have 
\begin{equation*}
\mathsf{P}_{\mathbb{N}}(T, \psi)=\sup\{\mathsf{P}_\mathcal{A}(T,\psi): 
\mathcal{A}\ \mathrm{is \ a \ finite\ subset\ of }\ \mathbb{N}\}.
\end{equation*}
\end{proposition}
Now we consider the specific potentials,
$$
\psi_{1}(x)=-s \log | T'(x)| - s \log B
$$
and
$$
\psi_{2}(x)=-s \log | T'(x)| - s \log B_1 + (1-s) \log B_2
$$
for some $1 < B,B_1,B_2 < \infty$ and $s\ge0 $. Note that if we let $B= \beta_0$ and $B_1=\beta_i, B_2 = \beta_{i-1}$ for $i=1,\ldots,m-1$, then we will have potential functions used in the formulation of the main result. It is clear that $\psi_{1}$ and $\psi_2$ satisfy the variation condition and then Proposition \ref{l5} holds. 

Thus, the pressure function \eqref{5} with potential $\psi_{1}$ is represented by 
\begin{align*}  \label{gddd}
\mathsf P_{\mathcal{A}} \Bigl(T, -s(\log B+\log |T^{\prime }(x)|) \Bigl)&=\lim_{n\to\infty}\frac{1}{n} \log \sum_{a_1,\ldots,a_n \in \mathcal{A}} e^{ S_n (-s(\log B+\log|T^{\prime }(x)|) ) }  \notag \\
&=\lim_{n\to\infty} \frac1n \log
\sum_{a_1,\ldots,a_n \in \mathcal{A}}\left(\frac1{B^{n}q_n^2}\right)^s,
\end{align*}
where we also used Proposition \ref{prop26}.
The last equality holds by
\begin{equation*}
{S_n (-s(\log B+\log |T^{\prime }(x)|)) }={-ns\log B-s\log q_{n}^{2}}.
\end{equation*}
which is easy to check by Proposition \ref{pp3}.
As before, we obtain the pressure function with potential $\psi_{2}$ by
\begin{align*}
\mathsf P_{\mathcal{A}}(T, -s\log|T^{\prime }(x)|-s\log B_{1}+(1-s)\log B_{2})&=\lim_{n\to\infty}
\frac1n \log \sum_{a_1,\ldots,a_n \in \mathcal{A}} e^{ S_n (-s\log|T^{\prime }(x)|-s\log B_{1}+(1-s)\log B_{2}) }  \notag \\
&=\lim_{n\to\infty} \frac1n \log
\sum_{a_1,\ldots,a_n \in \mathcal{A}}\left(\frac1{B_{1}^{n}q_n^2}\right)^sB_{2}^{(1-s)n}.
\end{align*}
For any $n\geq1$ and $s\geq0,$ we write
\begin{align*}\label{gnro}
f_{n}^{(1)}\left( s \right)&:=\sum_{a_{1},\ldots ,a_{n}\in \mathcal{A}}\frac{1}{
\left( B^{n}q_{n}^{2}\right) ^{s }} , \\
f_{n}^{(2)}\left( s \right)&:=\sum_{a_1,\ldots,a_n \in \mathcal{A}}\left(\frac1{B_{1}^{n}q_n^2}\right)^sB_{2}^{(1-s)n}.
\end{align*}
and denote
\begin{align*}
& s_{n,B}\left( \mathcal{A}\right) =\inf \left\{ s \geq 0:f_{n}^{(1)}\left( s
\right) \leq 1\right\}, \quad& & g_{n,B_{1},B_{2}}\left( \mathcal{A}\right) =\inf \left\{ s \geq 0:f_{n}^{(2)}\left( s
\right) \leq 1\right\}, \\
& s_{B}(\mathcal{A})=\inf \left\{s\geq 0 :\mathsf{P}_{\mathcal{A}}(T, \psi_{1})\le 0\right \}, \quad & &g_{B_{1},B_{2}}(\mathcal{A})=\inf \{s\geq 0 :\mathsf{P}_{\mathcal{A}}(T, \psi_{2})\le 0\}, \\
& s_{B}(\mathbb{N})=\inf \left\{s\geq 0 :\mathsf{P}(T, \psi_{1})\le 0 \right\}, \quad & &g_{B_{1},B_{2}}(\mathbb{N})=\inf \left\{s\geq 0 :\mathsf{P}(T, \psi_{2})\le 0\right\}.
\end{align*}
If $\mathcal{A}\in\mathbb{N}$ is a finite set, then by \cite{WaWu08} it is
straightforward to check that both $f_{n}^{(i)}\left( s\right) $ and $\mathsf{P}_{
\mathcal{A}}(T, \psi_{i})$ for $i=1,2$ are monotonically decreasing and
continuous with respect to $s$. Thus, $s_{n,B}\left( \mathcal{A}\right)$, $s_{B}(
\mathcal{A})$, $g_{n,B_{1},B_{2}}\left( \mathcal{A}\right)$ and $g_{B_{1},B_{2}}(
\mathcal{A})$ are, respectively, the unique solutions of $f_{n}^{(1)}\left( s\right)
= 1 $, $\mathsf{P}_{\mathcal{A}}(T, \psi_{1})= 0$, $f_{n}^{(2)}\left( s\right)
= 1 $ and $\mathsf{P}_{\mathcal{A}}(T, \psi_{2})= 0.$
For simplicity, when $\mathcal{A}=\left\{ 1,2,\ldots ,M
\right\}$ for some $M>0$, we write $s_{n,B}\left({M }\right) $ for $
s_{n,B}\left( \mathcal{A}\right) $, $s_{B}\left({M }\right) $ for $
s_{B}\left( \mathcal{A}\right) $, $g_{n,B_{1},B_{2}}\left({M }\right) $ for $
g_{n,B_{1},B_{2}}\left( \mathcal{A}\right) $ and $g_{B_{1},B_{2}}\left({M }\right) $ for $
g_{B_{1},B_{2}}\left( \mathcal{A}\right) $. When $\mathcal{A}=\mathbb{N}$, we write $s_{n,B}$ for $
s_{n,B}(\mathbb{N})$, $s_{B}$ for $
s_{B}(\mathbb{N})$, $g_{n,B_{1},B_{2}}$ for $
g_{n,B_{1},B_{2}}(\mathbb{N})$ and $g_{B_{1},B_{2}}$ for $
g_{B_{1},B_{2}}(\mathbb{N})$.
As a consequence, we have 
\begin{corollary} \label{cor2.1}For any integer $M\in\N,$
\label{p2} 
\begin{equation*}
\lim_{n\to \infty}s_{n,B}(M)=s_{B}(M), \ \ \lim_{M\to \infty}s_{B}(M)=s_{B}, \ \ \lim_{n\to \infty}g_{n,B_{1},B_{2}}(M)=g_{B_{1},B_{2}}(M),\ \ \lim_{M\to \infty}g_{B_{1},B_{2}}(M)=g_{B_{1},B_{2}},
\end{equation*}
where $s_B$ and $g_{B_1,B_2}$ are defined in \eqref{pres1} and \eqref{pres2} respectively.\\
Note that,  $s_{B}$ and $g_{B_{1},B_{2}}$ are continuous respectively as a function of $B$ and $B_{1},B_{2}$. Moreover,
$$\lim_{B\rightarrow 1}s_{B}=1,\ \ \lim_{B\rightarrow\infty}s_{B}=1/2.$$
\end{corollary}
\begin{proof}
The last two equations are proved in \cite[Lemma 2.6]{WaWu08} and others are consequences of Proposition \ref{l5}.
\end{proof}
As before, we will set $B= \beta_0$ and $B_1=\beta_i, B_2 = \beta_{i-1}$ for $i=1,\ldots,m-1$ in Corollary \ref{cor2.1}, so that $s_B$ and $g_{B_1,B_2}$ will become $d_0$ and $d_i$ with $i=1,\ldots,m-1$ respectively.

\section{Hausdorff dimension of $S_m(A_0,\ldots,A_{m-1})$.}\label{cantor}

The proof of Theorem \ref{General} consists of two parts, the upper bound and the lower bound.  
For notational simplicity, we take $c_0=\cdots=c_{m-1}=1$ and the other case can be done with appropriate modifications. That is, we will be dealing with the set
$$S_m(A_0,\ldots,A_{m-1}) = \left\{ x\in[0,1): \,  A_i^n \le a_{n+i}(x) < 2 A_i^n, 0 \le i \le m-1,\,\text{for infinitely many } n\in\N \right\}.$$

\subsection{Upper bound.} 
At first, for each $n \ge 1$ and $(a_1,\ldots, a_{n-1})\in \N^{n-1}$ define
\begin{align*}
F_n  & = \left\{ x\in[0,1): \,  A_i^n \le a_{n+i}(x) < 2 A_i^n, \, 0 \le i \le m-1 \right\}  \\
 &= \bigcup_{a_1,\ldots,a_{n-1}\in\N} \left\{ x\in[0,1): a_j(x)=a_j, 1\le j < n , A_i^n \le a_{n+i}(x) < 2 A_i^n,\, 0 \le i \le m-1        \right\} \\
 &: =   \bigcup_{a_1,\ldots,a_{n-1}\in\N}   F_n(a_1,\ldots, a_{n-1} ).
\end{align*}
Then
\begin{align*}
S_m(A_0,\ldots,A_{m-1}) &=\bigcap\limits_{N=1}^{\infty}  \bigcup\limits_{n=N}^{\infty} F_n  =  \bigcap\limits_{N=1}^{\infty}  \bigcup\limits_{n=N}^{\infty}  \bigcup_{a_1,\ldots,a_{n-1}\in\N}   F_n(a_1,\ldots, a_{n-1} ).
\end{align*}
There are $m$ potential optimal covers for $F_{n}$ for each $n\ge N$. Define
$$
J_{n-1}(a_1,\ldots,a_{n-1})=  \bigcup\limits_{A_0^n\leq a_n <2A_0^n } I_{n}(a_1,\ldots,a_{n}).
$$
Next, for any $1 \le i \le m-1$ and any $A_i^n \le a_{n+i}(x) < 2 A_i^n$, define
$$
J_{n-1+i}(a_1,\ldots,a_{n-1+i})=  \bigcup\limits_{A_i^n\leq a_{n+i} <2A_1^i } I_{n+i}(a_1,\ldots,a_{n+i}).
$$
Then, by using Proposition \ref{pp3} and Proposition \ref{pp2} recursively, for every $0 \le i \le m-1$ we obtain
\begin{align*}
|J_{n-1+i}(a_1,\ldots,a_{n-1+i})|& =  \sum\limits_{A_i^n\leq a_{n+i} <2A_i^n } \left| \frac{p_{n+i}}{q_{n+i}} - \frac{p_{n+i}+p_{n-1+i}}{q_{n+i}+q_{n-1+i}}\right| \\ &\asymp \frac{1}{A_i^n q_{n-1+i}^2} \\
&\asymp 
\frac{1}{\beta_i^n\cdot \beta_{i-1}^n q_{n-1}^2}.
\end{align*}
Therefore, for covering by $J_{n-1}$, for any $\varepsilon>0$, the $(d_0+2\varepsilon)$-dimensional Hausdorff dimension of $ S_m(A_0,\ldots,A_{m-1})$ can be estimated as
\begin{equation*}
\begin{split}
\HH^{d_0+2\varepsilon}( S_m(A_0,\ldots,A_{m-1}))& \leq \liminf_{N\to\infty}   \sum\limits_{n=N}^\infty   \sum\limits_{a_1,\ldots,a_{n-1}}  |J_{n-1}(a_1,\ldots,a_{n-1})|^{d_0+2\varepsilon} \\
& \leq  \liminf_{N\to\infty}   \sum\limits_{n=N}^\infty  \frac{1}{2^{(n-1)\varepsilon}} \sum\limits_{a_1,\ldots,a_{n-1}} \left(\frac{1}{\beta_0^n q_{n-1}^{2}}\right)^{d_0}\\
&\le \liminf_{N\to\infty}   \sum\limits_{n=N}^\infty  \frac{1}{2^{(n-1)\varepsilon}}<\infty,
\end{split}
\end{equation*}
where we used that $\beta_{-1}=1$. Hence, from the definition of Hausdorff dimension it follows that
\begin{equation}\label{0case}
\hdim S_m \le d_0.
\end{equation}
For coverings by $J_{n-1+i}$ for $1 \le i \le m-1$, we get that $(d_i+2\varepsilon)$-dimensional Hausdorff dimension of $ S_m(A_0,\ldots,A_{m-1})$ can be estimated as
\begin{equation*}
\begin{split}
\HH^{d_i+2\varepsilon}( S_m(A_0,\ldots,A_{m-1}))& \leq \liminf_{N\to\infty}   \sum\limits_{n=N}^\infty   \sum\limits_{a_1,\ldots,a_{n-1}}   \sum\limits_{\substack{ A_j^n\leq a_{n+j} <2A_j^n \\ \text{for all } 0\le j \le i-1}} |J_{n-1+i}(a_1,\ldots,a_{n-1+i})|^{d_i+2\varepsilon} \\
& \leq  \liminf_{N\to\infty}   \sum\limits_{n=N}^\infty   \sum\limits_{a_1,\ldots,a_{n-1}}  \sum\limits_{\substack{ A_j^n\leq a_{n+j} <2A_j^n \\ \text{for all } 0\le j \le i-1}} \left(\frac{1}{\beta_i^n\cdot \beta_{i-1}^n q_{n-1}^2}\right)^{d_i+2\varepsilon}\\
&\le \liminf_{N\to\infty}   \sum\limits_{n=N}^\infty   \sum\limits_{a_1,\ldots,a_{n-1}} \beta_{i-1}^n \left(\frac{1}{\beta_i^n\cdot \beta_{i-1}^n q_{n-1}^2}\right)^{d_i+2\varepsilon}\\
& =  \liminf_{N\to\infty}   \sum\limits_{n=N}^\infty   \sum\limits_{a_1,\ldots,a_{n-1}}  \frac{\beta_{i-1}^{(1-d_i)n}}{(\beta_i^n  q_{n-1}^2)^{d_i}} \frac{1}{(\beta_i^n\cdot \beta_{i-1}^n  q_{n-1}^2)^{2\varepsilon}}\\
&\le \liminf_{N\to\infty}   \sum\limits_{n=N}^\infty  \frac{1}{2^{(n-1)\varepsilon}}<\infty. 
\end{split}
\end{equation*}
Thus, the upper bound is obtained immediately by combining the latter with \eqref{0case}, so
$$
\hdim S_m \le  \min_{0 \le i \le m-1} d_i.
$$

\subsection{Lower bound.} In this subsection we will determine the lower bound for the dimension of $ S_m(A_0,\ldots,A_{m-1})$ by using the mass distribution principle (Proposition \ref{p1}).


For convenience, let us define some dimensional numbers. For any integers $N,M$ and $0 \le i \le m-1$ define the dimensional number $\dd_i=\dd_{i,N}(M)$ as the solution to
\begin{equation}\label{defdd}
\sum\limits_{1\leq a_1,\ldots,a_{N}\leq M} \frac{\beta_{i-1}^N}{((\beta_i \beta_{i-1})^{N} q_{N}^{2})^{\dd_i}}=1.
\end{equation}
More specifically, each equation has a unique solution and, by Corollary \ref{cor2.1}, 
$$
\lim\limits_{M\to\infty} \lim\limits_{N\to\infty} \dd_{i,N}(M) = d_i.
$$
Take a sequence of large sparse integers $\{\ell_k\}_{k\geq1}$, say, $\ell_k \gg e^{\ell_1+\dots+\ell_{k-1}}.$ For any $\varepsilon>0$, choose integers $N,M$ sufficiently large such that
$$
 \dd_i>d_i-\varepsilon, \qquad \left(2^{(N-1)/2})\right)^{\varepsilon/2}\geq 2^{100}.
$$
Let 
\begin{equation}\label{sequence}
n_k-n_{k-1}=\ell_kN + m, \, \forall k\geq1,
\end{equation}
such that 
$$\left( 2^{\ell_k(N-1)/2} \right)^{\frac{\varepsilon}{2}} \ge  \prod\limits_{t=1}^{k-1} (M+1)^{\ell_t N}(\beta_{m-1})^{\sum_{i=1}^{t}\ell_{i}N+t}.$$

At this point, define a subset of $ S_m(A_0,\ldots,A_{m-1})$ as
\begin{align}
E = E_{M,N}=  \{ x\in[0,1): \,   A_i^{n_k} \le a_{n_k+i}(x) < 2 A_i^{n_k} &\text{ for all } k\geq1, \text{ for all } 0\le i \le m-1 \label{mainsubset}\\
& \text{ and } a_n(x)\in\{1,\ldots,M\} \text{ for other } n\in\N \}.\notag
\end{align}

Next we proceed to make use of a symbolic space.  Define $D_0=\{\emptyset\}$, and for any $n\geq1$, define 
 \begin{align*}
    D_n=\Bigg\{(a_1,\ldots, a_n)\in \N^n: A_i^{n_k}\le a_{n_k+i}&<2A_i^{n_k}, \ {\text{for all}} \ 0\le i\le m-1, k\ge 1 \ {\text{with}} \ {n_k+i}\le n;\\ &{\text{and}}\ a_j\in \{1,\ldots, M\}, \ {\text{for other $j\le n$}}\Bigg\}.
  \end{align*} This set is just the collection of the prefixes of the points in $E$.
Moreover, the collection of finite words of length $N$ is denoted by
$$
\UU = \{ w= (\sigma_1,\dots, \sigma_N): 1\leq \sigma_i \leq M, 1\leq i\leq N\}
$$
and, for the remainder of the paper we always use $w$ to denote an element from $\UU$. 

\subsubsection{Cantor structure of $E$.}
In this subsection, we depict the structure of $E$ with the help of symbolic space as mentioned above. For any $(a_1,\ldots, a_n)\in D_n$, define $$
J_n(a_1,\ldots,a_n)=\bigcup_{a_{n+1}: (a_1,\ldots,a_n, a_{n+1})\in D_{n+1}}I_{n+1}(a_1,\ldots,a_n, a_{n+1})
$$ and call it a {\em basic cylinder} of order $n$. More precisely, for any $k\ge 0$\begin{itemize}

\item when $n_k+m-1 \le n<n_{k+1}-1$ (by viewing $n_0=0$), $$
J_n(a_1,\ldots,a_n)=\bigcup_{1\le a_{n+1}\le M}I_{n+1}(a_1,\ldots,a_n, a_{n+1}).
$$
\item when $n=n_{k+1}+i-1$ for some $0\le i \le m-1$,
$$
J_n(a_1,\ldots,a_n)=\bigcup_{A_i^{n_{k+1}}\le a_{n+1}< 2 A_i^{n_{k+1}}}I_{n+1}(a_1,\ldots,a_n, a_{n+1}).
$$
\end{itemize}
Then we define the level $n$ of the Cantor set $E$ as
$$
\mathcal{F}_n=\bigcup_{(a_1,\ldots,a_n)\in D_n}J_n(a_1,\ldots,a_n).
$$ 
Consequently, the Cantor structure of $E$ is described as follows
$$
E=\bigcap_{n=1}^{\infty}\mathcal{F}_n=\bigcap_{n=1}^{\infty}\bigcup_{(a_1,\ldots,a_n)\in D_n}J_n(a_1,\ldots,a_n).
$$

We observe that every element $x\in E$ can be written as \begin{align*}
x=[w_1^{(1)},\ldots, w_{\ell_1}^{(1)}, a_{n_1},\ldots a_{n_1+m-1}, & w_1^{(2)},\ldots, w_{\ell_2}^{(2)}, a_{n_2},\ldots, a_{n_2+m-1},
\ldots,  w_1^{(k)},\ldots, w_{\ell_k}^{(k)}, a_{n_k},\ldots,a_{n_k+m-1},\ldots],
\end{align*} where $$
w^{(p)}_k\in \UU, \ {\text{and}}\ \ A_{i}^{n_k}\le a_{n_k+i}\le 2A_{i}^{n_k}, \ \text{ for all }  k,p\in\mathbb{N} \ \ 0\le i\le m-1.
$$
Then the length of cylinder set can be estimated as follows.

\begin{lemma}[Length estimation] Let $x\in E$ and $n_{k-1}+m-1 \leq n < n_{k}+m-1$ for some $k\ge 1$.
\begin{itemize}
\item for $n_{k-1}+m-1 \leq n < n_k -1 = n_{k-1} + k + \ell_k N$,
$$
\frac{1}{8q_n^2} \leq | J_n(x) | \leq \frac{1}{q_n^2}.
$$
\item for $n=n_k-1+i$ for some $0\le i \le m-1$, i.e. for $n_k \leq n \leq n_k+m-2$,
$$
| J_{n_k-1+ i}(x) | \geq \frac{1}{6\cdot 4^{i} \beta_{i}^{n_k} \beta_{i-1}^{n_k} q_{n_k-1}^2 } \geq  \frac{1}{6\cdot 4^{i} \beta_{i}^{n_k} \beta_{i-1}^{n_k}}  \left( \prod_{i=1}^{\ell_{k}}\frac{1}{q_N(w_i^{(k)})}\right)^{2(1+\varepsilon)}.
$$
\item for $n=n_k +m-1$,
$$
| J_{n_k+m-1}(x)| \ge   \frac{1}{6\cdot 4^{i} A_{m-1}^{n_k} \beta_{m-1}^{n_k} \beta_{m-2}^{n_k} q_{n_k-1}^2 }.
$$
  \item for each $1\le \ell<\ell_{k+1}$,
 \begin{align}\label{ff12}|J_{n_k+m-1+\ell N}(x)|&\ge \frac{1}{2^3}\cdot \left(\frac{1}{2^{2\ell}}\cdot \prod_{i=1}^{\ell}\frac{1}{q_N^2(w_i^{(k+1)})}\right)\cdot \frac{1}{q^2_{{n_k}+1}} \ge  \left( \prod_{i=1}^{\ell}\frac{1}{q_N^2(w_i^{(k+1)})}\right)^{1+\varepsilon}\cdot \frac{1}{q^2_{{n_k}+1}}. \end{align} 
\,
\item for $n_k+1+(\ell-1)N<n<n_k+1+\ell N$ with $1\le \ell\le \ell_{k+1}$, \begin{align}\label{ff13}
|J_{n}(x)|\ge c\cdot |J_{n_k+1+(\ell-1)N}(x)|,
\end{align} where $c=c(M, N)$ is an absolute constant.
\end{itemize}
\end{lemma}

\subsection{Mass distribution}\
In this subsection, we define $m$ mass distributions along the basic intervals $J_n(x)$ containing $x$. These mass distributions then can be extended respectively into $m$ probability measures supported on $E$ by the Carath\'eodory extension theorem. 
Now let us distribute the measure by induction.
For $n\le n_1+1$, 
\begin{enumerate}
  \item when $n=\ell N$ for each $1\le \ell\le \ell_1$, for $0 \le j \le m-1$  define
$$
  \mu_j(J_{\ell N}(x))=\prod_{i=1}^{\ell}\frac{\beta_{j-1}^N}{q_N(w_i^{(1)})^{2\dd_j}\cdot (\beta_j \beta_{j-1})^{\dd_j N}}.
  $$
  We note that measures $\mu_j$ for $0 \le j \le m-1$ can be defined on all basic cylinders of order $\ell N$ since  $x$ is arbitrary.  
  \item when $(\ell-1)N<n<\ell N$ for some $1\le \ell\le \ell_1$ and for all $0 \le j\le m-1$, define
  $$
  \mu_j(J_n(x))=\sum_{J_{\ell N}\subset J_n(x)}\mu_j(J_{\ell N}(x)).
  $$ The consistency property as mentioned above fulfills the measure of other basic intervals of order less than $n_{1}-1$.
        \item when $n=n_{1}+i$ for each $0\le i\le m-1$ and $0 \le j\le m-1$, define $$
  \mu_j(J_{n_1+i}(x))=\prod_{k=0}^i\frac{1}{A_k^{n_1}}\cdot \mu_j(J_{n_1-1}(x)) = \frac{1}{\beta_i^{n_1}} \cdot \mu_j(J_{n_1-1}(x)) .
  $$
  To make the proof more consistent, let us note that 
$$
 \mu_j(J_{n_1-1}(x)) = \frac{1}{(\beta_{-1})^{n_1}}  \mu_j(J_{n_1-1}(x)).
$$
\end{enumerate}
Assume the measure of all basic intervals of order $n$ has been defined when $n_{k}+m-1<n\le n_{k+1}+m-1$. 
\begin{enumerate}
  \item  When $n=n_{k}+m-1+\ell N$ for each $1\le \ell\le \ell_{k+1}$, for $0 \le j \le m-1$ define 
\begin{align}\label{ff777}
  \mu_j(J_{n_{k}+m-1+\ell N}(x))=\prod_{i=1}^{\ell}\frac{\beta_{j-1}^N}{q_N(w_i^{(k+1)})^{2\dd_j}\cdot  (\beta_j \beta_{j-1})^{\dd_j N}} \mu_j(J_{n_{k}+m-1}(x)) .
  \end{align}
  \item When $n_k+m-1+(\ell-1)N<n<n_k+m-1+\ell N$ for some $1\le \ell\le \ell_1$ and for $0 \le j\le m-1$, define
  $$
  \mu_j(J_n(x))=\sum_{J_{n_k+m-1+\ell N}\subset J_n(x)}\mu_j(J_{n_k+m-1+\ell N}).
  $$
Furthermore, for each measure, compared with the measure of a basic cylinder of order $n_k+1+(\ell-1) N$ and its offsprings of order $n_k+1+\ell N$, there is only a multiplier between them. More precisely, for measure $\mu_j$ it is the term 
$$
\frac{\beta_{j-1}^N}{q_N^{2\dd_j}(w_{\ell}^{(k+1)}) (\beta_j \beta_{j-1})^{\dd_j N}}.
$$ 
Thus for each $0 \le j\le m-1$, there is an absolute constant $\hat{c}>0$ such that
$$\mu_j(J_n(x))\ge \hat{c} \cdot \mu_j\Big(J_{n_k+m-1+(\ell-1)N}(x)\Big),$$
since the above two terms are uniformly bounded.  
  \item when $n=n_{k+1}+i$ for each $0\le i\le m-1$ and $0 \le j\le m-1$, define 
\begin{align}\label{ff11}
  \mu_j(J_{n_{k+1}+i}(x))&=\prod_{j=0}^i\frac{1}{A_j^{n_{k+1}}}\cdot \mu_j(J_{n_{k+1}-1}(x)) = \frac{1}{\beta_i^{n_{k+1}}} \cdot \mu_j(J_{n_{k+1}-1}(x)).
  \end{align}
  \item As for other orders of measure, to ensure the consistency property, let their measure equal to the summation of the measure of their offsprings. For each integer $n$, the relation between measures of a basic cylinder and its predecessor acts like the case $n_{k+1}+m-1+(\ell-1)N<n<n_{k+1}+m-1+\ell N$, for each $0 \le j\le m-1$, there is a constant $\hat{c}>0$ such that 
\begin{align}\label{ff2}
\mu_j(J_{n+1}(x))\ge \hat{c} \cdot \mu_j(J_n(x)).
\end{align}
\end{enumerate}

\subsection{H\"{o}lder exponent of $\mu$ for basic cylinders}\
We need to compare the measure and length of $J_{n}(x)$. Recall the definition \eqref{defdd} of $\dd_i$. For each $N,M$ we can arrange $\dd_i$'s in the ascending order. Note that if $\dd_j \le \dd_k$, then we have
$$
\sum\limits_{1\leq a_1,\ldots,a_{N}\leq M} \frac{1}{ q_{N}^{2\dd_j}}\ge \sum\limits_{1\leq a_1,\ldots,a_{N}\leq M} \frac{1}{ q_{N}^{2\dd_k}}
$$
and by definition \eqref{defdd} we get
$$
\frac{\beta_{j-1}^{1-{\dd_j}}}{\beta_j^{\dd_j}} \le \frac{\beta_{k-1}^{1-{\dd_k}}}{\beta_k^{\dd_k}} .
$$
Once again using the fact that $\dd_j \le \dd_k$, we obtain
\begin{equation}\label{conditions}
\frac{\beta_{j-1}^{1-{\dd_j}}}{\beta_j^{\dd_j}}  \le \frac{\beta_{k-1}^{1-{\dd_j}}}{\beta_k^{\dd_j}}.
\end{equation}
Thus, if $\dd_j =  \min_{0 \le k \le m-1} \dd_k$, then \eqref{conditions} holds for any $0 \le k \le m-1$. For every $0 \le j \le m-1$, will use measure $\mu_j$ when $\min_{0 \le k \le m-1} \dd_k = \dd_j $.

\begin{enumerate}
\item When $n=n_k-1+i$ for $0 \le i \le m-1$. For every $0 \le j \le m-1$ and for every $0 \le i \le m-1$, we have
\begin{align*}
\mu_j (J_{n_{k}-1+i}(x))&\le \frac{1}{ \beta_{i-1}^{n_k}} \cdot  \frac{\beta_{j-1}^{n_k}}{(\beta_{j} \beta_{j-1})^{\dd_j n_k}} \prod_{i=1}^{\ell_k}\frac{1}{q_N(w_i^{(k)})^{2 \dd_j }} \\
 & \le \frac{1}{(\beta_{i} \beta_{i-1})^{n_k \dd_j}} \left(  \frac{1}{ q_{n_k-1}^2}\right)^\frac{\dd_j}{1+\varepsilon} 
 \le  \left(  \frac{1}{(\beta_{i} \beta_{i-1})^{n_k} q_{n_k-1}^2}\right)^\frac{\dd_j}{1+\varepsilon}  \\
& \le \hat{c}_1 \cdot |J_{n_{k}-1+i}(x) |^{\frac{\dd_j}{1+\varepsilon}}.
\end{align*}

\item When $n=n_k +m-1$. For every $0 \le j \le m-1$ we have
\begin{align*}
\mu_j (J_{n_{k}+m-1}(x)) &= \frac{1}{A_{m-1}^{n_k}} \mu_j (J_{n_{k}+m-2}(x)) \le  \frac{1}{A_{m-1}^{n_k}}\cdot  \hat{c}_1 \cdot \left(    \frac{1}{A_{m-1}^{n_k} q_{n_k+m-2}^2       } \right)^{\frac{\dd_j}{1+\epsilon}}\\
& \le  \hat{c}_1 \cdot \left(    \frac{1}{A_{m-1}^{2n_k} q_{n_k+m-2}^2       } \right)^{\frac{\dd_j}{1+\epsilon}} \le \hat{c}_2  |J_{n_{k}+m-1}(x) |^{\frac{\dd_j}{1+\varepsilon}} \le \hat{c}_2  \left(    \frac{1}{ q_{n_k+m-1}^2  } \right)^{\frac{\dd_j}{1+\epsilon}}.
\end{align*}

\item When $n=n_k+m-1 +\ell N$ for some $1\le \ell < \ell_{k+1}$. Then for each $0 \le j \le m-1$,
\begin{align*}
\mu_j(J_{n_k+m-1+\ell N}(x))
  \le  \hat{c}_2   \prod_{i=1}^{\ell}\frac{1}{q_N(w_i^{(k+1)})^{2s}}  \left( \frac{1}{ q_{n_k+m-1}^2} \right)^\frac{{\dd_j}}{1+\varepsilon}\le \hat{c}_2 | J_{n_k+m-1+\ell N}(x)|^\frac{{\dd_j}}{1+\varepsilon}. 
\end{align*}

\item For other $n$, let $1 \le \ell \le \ell_k$ be the integer such that
$$
n_k+m-1+(\ell-1)N \le n < n_k+m-1 +\ell N.
$$
Recall \eqref{ff13}. Then for each $0 \le j\le m-1$,
\begin{align*}
  \mu_j(J_n(x))&\le \mu_j(J_{n_k+m-1+(\ell-1)N}(x)) \le \hat{c}_2 \cdot \big|J_{n_k+m-1+(\ell-1)N}(x)\big|^\frac{{\dd_j}}{1+\varepsilon}\le \hat{c}_2 \cdot c \cdot \big|J_{n}(x)\big|^\frac{{\dd_j}}{1+\varepsilon}.
\end{align*}
\end{enumerate}
\,
In a summary, we have shown that for some absolute constant $c_3$, for any $n\ge 1$ and $x\in E$, \begin{align}\label{g1}
  \mu_j(J_n(x))\le \hat{c}_3\cdot |J_n(x)|^\frac{{\dd_j}}{1+\varepsilon} \leq \hat{c}_3\cdot |J_n(x)|^\frac{{\min_{0\le j \le m-1}{\dd_j}}}{1+\varepsilon} .
\end{align}

\subsection{H\"{o}lder exponent for a general ball}

For simplicity, write
$$
\tau= \frac{\min_{0\le j \le m-1}{\dd_j}}{1+\varepsilon}.
$$
The next lemma gives a minimum gap between two adjacent fundamental cylinders.
\begin{lemma}[Gap estimation]\label{l3} Denote by $G_n(a_1,\ldots, a_n)$ the gap between $J_n(a_1,\ldots, a_n)$ and other basic cylinders of order $n$. Then $$
G_n(a_1,\ldots, a_n)\ge \frac{1}{M}\cdot |J_n(a_1,\ldots,a_n)|.
$$
\end{lemma}
\begin{proof}
  The proof of this lemma is derived from the positions of the cylinders in Proposition \ref{pp2}. We omit the details and refer the reader to its analogous proof in \cite[Lemma 5.3]{HuWuXu}.
\end{proof}

Then for any $x\in E$ and $r$ small enough, there is a unique integer $n$ such that 
$$G_{n+1}(x)\le r<G_{n}(x).$$ 
This implies that the ball $B(x,r)$ can only intersect one basic cylinder $J_n(x)$, and so all the basic cylinders of order $n+1$ for which $B(x,r)$ can intersect are all contained in $J_n(x)$. Note that $n_{k-1}+1\le n <n_{k}+1$. 
\begin{enumerate}
\item For $n_{k-1}+1\le n<n_{k}-1$, by (\ref{ff2}) and (\ref{g1}), it follows that for each $0 \le j\le m-1$,
\begin{align*}
\mu_j(B(x,r))&\le \mu_j(J_n(x))\le c\cdot \mu_j(J_{n+1}(x))\le c\cdot \hat{c}_3\cdot \big|J_{n+1}(x)\big|^{\tau}\\&\le \hat{c} \cdot \hat{c}_3\cdot M\cdot (G_{n+1}(x))^{\tau}\le \hat{c} \cdot \hat{c}_3\cdot M\cdot  r^{\tau}.
\end{align*}
\item For $n=n_{k}-1+i$ for $0 \le i \le m-1$, the ball $B(x,r)$ can only intersect one basic cylinder $J_{n_k-1+i}(x)$ of order $n_k-1+i$. Next, the number of basic cylinders of order $n_{k}+i$ which are contained in $J_{n_k-1+i}(x)$ and intersect the ball can be calculated as follows.
We write $x=[a_1(x),a_2(x),\ldots]$ and observe that any basic cylinder $J_{n_{k}+i}(a_{1},\ldots,a_{n_k+i})$ is contained in the cylinder $I_{n_k+i}(a_{1},\ldots,a_{n_k+i})$. Note that $A_{i}^{n_{k}}\le a_{n_{k}+i}\le 2A_{i}^{n_{k}}$, the length of cylinder $I_{n_k+i}$ is
$$
\frac{1}{q_{n_k+i}(q_{n_k+i}+q_{n_k-1+i})}\ge \frac{1}{2^{5}}\cdot \frac{1}{q_{n_k-1+i}^2A_i^{2n_k+i}}.
$$
We also note that radius $r$ is sometimes too small to cover a whole cylinder of order $n_{k}+i$. The exposition needs to split into two parts. When 
$$r<\frac{1}{2^{5}} \frac{1}{q_{n_k-1+i}^2A_i^{2n_k+i}},$$ 
then the ball $B(x,r)$ can intersect at most three cylinders $I_{n_k+i}(a_1,\ldots,a_{n_{k+i}})$ and so three basic cylinders $J_{n_k+i}(a_1,\ldots,a_{n_{k}+i})$. Since each measure has the same distribution on these intervals, for $0 \le j\le m-1$,
\begin{align*}
  \mu_j(B(x,r))&\le 3\mu_j(J_{n_k+i}(x))\le 3 \cdot \hat{c}_3\cdot |J_{n_k+i}(x)|^{\tau}\\ 
&\le 3\cdot \hat{c}_3\cdot M\cdot G_{n_k+i}(x)^{\tau}\le 3\cdot \hat{c}_1\cdot M\cdot r^{\tau}.
  \end{align*}
When $$
  r\ge \frac{1}{2^{5}} \frac{1}{q_{n_k-1+i}^2A_i^{2n_k+i}},
  $$ then the number of basic cylinders of order $n_k+i$ for which the ball $B(x,r)$ can intersect is at most $$
  {2^{6}r}\cdot q_{n_k-1+i}^{2}A_i^{2n_{k}+i}+2\le 2^7\cdot {r}\cdot q_{n_k-1+i}^{2}A_i^{2n_{k}+i}.$$ 
Thus, for $0\le j\le m-1$,
 \begin{align*}
    \mu_j(B(x,r))&\le \min\Big\{\mu_j(J_{n_k-1+i}(x)),\ \  2^7\cdot {r}\cdot q_{n_k-1+i}^{2}(u)A_i^{2n_{k}+i}\cdot \frac{1}{A_i^{n_k+i}}\cdot \mu_j(J_{n_k-1+i}(x))\Big\}\\
    &\le \hat{c}_3\cdot |J_{n_k-1+i}|^{\tau}\cdot \min\Big\{1, 2^7\cdot {r}\cdot q_{n_k-1+i}^{2}(u)A_i^{n_{k}+i}\Big\}\\
    &\le \hat{c}_3\cdot \left(\frac{1}{q_{n_k-1+i}^2 A_i^{n_k+i}}\right)^{\tau} \cdot \Big(2^7\cdot {r}\cdot q_{n_k-1+i}^{2}A_i^{n_{k}+i}\Big)^{\tau(1-\epsilon)}\\
    &=\hat{c}_4 \cdot r^{\tau}.
  \end{align*}

\end{enumerate}

\subsection{Conclusion}
Thus by applying the mass distribution principle (Proposition \ref{p1})
\begin{equation*}\label{dim min}
\hdim  S_m(A_0,\ldots,A_{m-1}) \ge \hdim E \ge \frac{\min_{0 \le i \le m-1} \dd_i}{1+\varepsilon}.
\end{equation*}

Since $\varepsilon>0$ is arbitrary, by letting $N\to\infty$ and then $M\to\infty$, we arrive at
$$
\hdim  S_m(A_0,\ldots,A_{m-1}) \geq \min_{0 \le i \le m-1} d_i.
$$
This finishes the proof.

\subsection{Remark on Theorem \ref{General}}\label{remark1}
 Note that in order to prove a lower bound for the Hausdorff dimension of $S_m(A_0,\ldots,A_{m-1})$, we have considered a subset of this set for which the location of the blocks of exponentially growing partial quotients is given by some large sparse integers sequence of a specific type. Namely, we required in \eqref{sequence} that a number of partial quotients between two blocks of exponentially growing partial quotient is a multiple of $N$. However, in some applications it is useful to have a result for sequences of less restricted type. To prove such a result, there is a little bit extra work to be done but the main idea of the construction is the same.

Consider an arbitrary sparse integer sequence $\{n_k\}_{k\geq1}$, express it in the form
$$
n_1 = \ell_1N+(N+r_1) \text{ and } n_{k+1} - n_k = m + \ell_{k+1} +(N+ r_{k+1}) \text{ for all } k\geq 1,
$$
where $0 \leq r_k < N$ for all $k\geq 1$. Denote $m_k = n_{k-1} + m + \ell_k N$ with $n_0 = -m$.
Consider a set 
\begin{align}
\hat{E}&=\hat{E}_M^N(A_0,\ldots,A_{m-1}) =  \{ x\in[0,1): \,   A_i^{n_k} \le a_{n_k+i}(x) < 2 A_i^{n_k} \text{ for all } k\geq1, \text{ for all } 0\le i \le m-1, \label{mainsubset1}\\
& a_{m_k+1}=\cdots=a_{n_k-1}=2  \text{ and } a_n(x)\in\{1,\ldots,M\} \text{ for other } n\in\N \}.\notag
\end{align}
This means that we set $N+r_{k}$ partial quotients prior to the beginning of each block of exponentially growing partial quotients to be equal to $2$. Now for this set the proof is exactly the same as in Theorem \ref{General}, except we will have to deal with those new partial quotients. This can be done by defining our measures for each $0 \leq j \leq m-1$ and for all $m_k < n < n_k$ as
$$
\mu_j ( J_n ) = \mu_j ( J_{m_k}).
$$
In the end, we will get
$$
\hdim \hat{E} \geq  \frac{\min_{0 \le i \le m-1} \dd_i}{1+\varepsilon}.
$$

\section{Applications}
Our main result is very helpful for obtaining lower bounds in different setups, which is usually the hardest part in determining the Hausdorff dimension of the underlying set. Here we give some examples both of known results, for which our Theorem \ref{General} could have been used to derive the same result, as well as a new problem, where our set also helps to derive an optimal lower bound.
\subsection{Known results}

\subsubsection{Wang-Wu Theorem, \cite[2008]{WaWu08}}
The most obvious example is a well-known theorem by Wang-Wu from \cite{WaWu08}. The authors were concerned with the set
$$
F(B) = \{ x\in [0,1) : a_n(x) \ge B^n  \ \text{\,for infinitely many } n\in\N \}. 
$$
Their main result about this set is
\begin{theorem}[{\protect\cite[Wang-Wu]{WaWu08}}]\label{WaWu}
For any $1\leq B<\infty$,
$$
 \hdim F(B) =s(B):=\inf \{s\geq 0 :\mathsf{P}(T, -s(\log B+\log |T^{\prime}|))\le 0\}.
$$
\end{theorem}
To get the optimal lower bound for this setup using Theorem \ref{General}, one can simply let $m=1$, $A_0= B$ and $c_0=1$, that is to consider the set
$$
S_1( B) = \{ x\in[0,1) : B^n \leq a_n(x) < 2B^n   \ \text{\,for infinitely many } n\}.
$$
By Theorem \ref{General}, 
\begin{align*}
\hdim S_1 (B) &= d_0 = \inf\{ s\ge0 : P( T, -s \log | T'| - s \log \beta_0 + (1-s) \log \beta_{-1})\le 0\}  \\
& = \inf\{ s\ge0 : P( T, -s \log | T'| - s \log B )\le 0\},
\end{align*}
which coincides with the result from Theorem \ref{WaWu}.\\

\subsubsection{Bakhtawar-Bos-Hussain Theorem, \cite[2020]{BBH2}} Recall that Kleinbock-Wadleigh showed that  the set $\EE_2(\Phi)$ has connections with the set of Dirichlet non-improvable numbers. In \cite{BBH2} authors have considered the set
\begin{equation*}\label{pqc}
 \mathcal{F}(B ):=\EE_2(B)\setminus\EE_1(B)=\left\{ x\in [
0,1):
\begin{array}{r}
a_{n+1}(x)a_{n}(x)\geq B^n \text{ for infinitely many }n\in 
\mathbb{N}
\text{ and} \\ 
a_{n+1}(x)<B^n \text{ for all sufficiently large }n\in 
\mathbb{N}
\end{array}
\right\}.  
\end{equation*} 
They proved that the difference set $ \mathcal{F}(B )$ has positive Hausdorff dimension. More precisely they proved
\begin{theorem}
\begin{equation}\label{thBBH}
\hdim  \mathcal{F}(B ) =t_B: =  \inf\{ s\ge0 : P( T, -s \log | T'| - s^2 \log B )\le 0\}.
\end{equation}
\end{theorem}
To get a lower bound in this setup using our result, we take an arbitrary sparse sequence $n_k$ and set $m=2, A_0 = B^{t_B}, A_1 = B^{1-t_B}$ for the set $\hat{E}$ from Section \ref{remark1},  that is, we consider the set
\begin{equation*}
\hat{E} (  B^{t_B},  B^{1- t_B}) =\left\{x\in[0,1):
\begin{split} 
 &\begin{split}& B^{n_k t_B} \leq a_{n_k}(x) < 2 B^{n_k t_B} \\
 & B^{n_k(1- t_B)} \leq a_{n_k+1}(x) < 2 B^{n_k (1- t_B)}
 \end{split}&\text{for all } k\in\N, \quad a_m \leq M \text{ for other } m \\
\end{split}
\right\},
\end{equation*}
which is clearly a subset of $ \mathcal{F}(B )$. We can see that by the choice of the parameters $d_0=d_1$, and so by the proof of Theorem \ref{General} and the remark in Section \ref{remark1}, we have
$$
\hdim \mathcal{F}(B) \geq \hdim \hat{E} (  B^{t_B},  B^{1- t_B}) \geq d_0 = \inf\{ s\ge0 : P( T, -s \log | T'| - s^2 \log B )\le 0\},
$$
which coincides with the lower bound from \eqref{thBBH}.

 \subsubsection{Hussain-Li-Shulga Theorem, \cite[2022]{HLS}} Theorem \ref{General} is also a direct generalisation of a Theorem 1.7 from \cite{HLS} that gives the Hausdorff dimension of the set 
$$
E(A_1,A_2) \defeq \left\{ x\in[0,1): \, c_1 A_1^n\leq a_n(x) <2c_1A_1^n,  \, \,  c_2 A_2^n\leq a_{n+1}(x)<2  c_2 A_2^n,\,\text{for infinitely many }n\in\N \right\}.
$$
\begin{theorem}\label{dim E(A_1,A_2)}
For any $A_1>1$,
$$\hdim E(A_1,A_2) = \min \left\{ s(A_1) , g_{(A_1 A_2),A_1} \right \},$$
where 
\begin{equation}\label{pres1}
s(A_1)=\inf \{s\geq 0 :\mathsf{P}(T, -s(\log A_1 +\log |T^{\prime}|))\le 0\}
\end{equation}
and 
\begin{equation}\label{pres2}
g_{(A_1 A_2),A_1} = \inf \{ s\geq0:P(T,-s\log |T'| -s\log A_1 A_2 +(1-s)\log A_1) \leq 0\}.
\end{equation}
\end{theorem}
One can easily see that this is a special case of our Theorem \ref{General} for $m=2$. We also should note that in \cite{HLS} this theorem was used to get a lower bound for the Hausdorff dimension of the set
\begin{equation*}
\FF_{B_1,B_2}=\left\{x\in[0,1):
\begin{split} 
 a_n(x)a_{n+1}(x) & \geq B_1^n \text{\,\, for infinitely many } n\in\N \\
 a_{n+1}(x) & < B_2^n  \text{\,\, for all sufficiently large } n\in\N
\end{split}
\right\},
\end{equation*}  
and as a corollary also for the set
\begin{equation*}
\FF(\Phi_1,\Phi_2) = \left\{x\in[0,1):
\begin{split} 
 a_n(x)a_{n+1}(x) & \geq\Phi_1(n) \text{\,\, for infinitely many } n\in\N \\
 a_{n+1}(x) & <\Phi_2(n)  \text{\,\, for all sufficiently large } n\in\N
\end{split}
\right\},
\end{equation*}
where $\Phi_i:\N\to(0,\infty)$ are any functions such that $\lim\limits_{n\to\infty} \Phi_i(n)=\infty$. 

 \subsubsection{Huang-Wu-Xu Theorem, \cite[2020]{HuWuXu}} 
Another example is the main result of Huang-Wu-Xu paper \cite{HuWuXu}. They have considered a set
$$
E_m (B) := \{ x\in[0,1) : a_n(x) a_{n+1}(x) \cdots a_{n+m-1} (x) \ge B^n \  \text{for infinitely many } n\in\N \}.
$$
At the heart of their paper is the following result.
\begin{theorem}\label{HWXtheorem}
For $1\leq B<\infty$, and any integer $m\ge1$,
\begin{equation}\label{hwxconst}
\hdim E_m(B) = t_B^{(m)} = \inf \{ s: \mathsf{P}(T, -f_m(s) \log B -s \log | T^{\prime} | ) \leq 0 \},
\end{equation}
where $f_m(s)$ is given by the following iterative formula:
$$
f_1(s)=s, \,\,\,\,\, f_{k+1}(s) = \frac{s f_k (s)}{1-s+f_k(s)}, \,\, k \geq 1.
$$
\end{theorem}
Denote $t= t_B^{(m)}$. To get a lower bound in this setup, one should take the set $S_m(A_0,\ldots,A_{m-1})$ from Theorem \ref{General} and let 
$$
A_i = B^{\frac{t^{m-1-i}(2t-1)(1-t)^i}{t^{m}-(1-t)^{m}}}, \,\,  0 \le i \le m-1.
$$
One can easily check that with this choice of parameters $d_0=\ldots =d_{m-1}$, and so for this particular set of parameters we have
$$
\hdim S_m(A_0,\ldots,A_{m-1}) = d_0 = \inf \{ s: \mathsf{P}(T, -f_m(s) \log B -s \log | T^{\prime} | ) \leq 0 \},
$$
which coincides with the result from Theorem \ref{HWXtheorem}.
 \subsubsection{Bakhtawar-Hussain-Kleinbock-Wang Theorem, \cite[2022]{BHKW}} 
 Same construction was also used in \cite{BHKW} for the set with the weighted product of two partial quotients. For any $t_0, t_1\in\R_{>0}$, consider the set
\begin{equation*}\label{weight}
\D_2^{\mathbf t}(B):=\left\{x\in[0, 1):  a_{n}^{t_0} a_{n+1}^{t_1} \ge B^n \ {\text{for infinitely many}} \ n\in \N \right\}.
\end{equation*}
The Hausdorff dimension of this set is given by
\begin{theorem}\label{bhkw}
\begin{equation*}
\hdim \D_2^{\mathbf t}(B )= \SS = \inf \{s\ge 0: P(-s\log |T'|-f_{t_0,t_1 }(s)\log B)\le 0\},
\end{equation*}
where $$f_{t_0}(s)=\frac{s}{t_0}, \ \ f_{t_0, t_1}(s)=\frac{sf_{t_0}(s)}{t_1[f_{t_0}(s)+\max\{0, \frac{s}{t_1}-\frac{2s-1}{t_0}\}]}.$$
\end{theorem}
 As in their paper, we separately consider the case, where the value of the maximum in denominator is $0$, so  when
$$
\frac{\SS}{t_1}-\frac{2\SS-1}{t_0}\le 0,
 $$
 one should simply consider the subset
 $$
\Big\{x\in [0,1): a_{n+1}^{t_1}(x)\ge B^n,\  \text{for infinitely many } n\in \N\Big\}
$$
of $\mathcal{D}_{2}^{\bold{t}}(B)$ to get the desired $$
\hdim \mathcal{D}_{2}^{\bold{t}}(B)\ge \SS.
$$
When $$
\frac{\SS}{t_1}-\frac{2\SS-1}{t_0}> 0,
$$ 
in a small neighborhood of $\SS$ we  have
 \begin{equation}
f_{t_0,t_1}(s)=\frac{sf_{t_0}(s)}{t_1\big[f_{t_0}(s)+\frac{s}{t_1}-\frac{2s-1}{t_0}\big]}.
\end{equation}
So to get the optimal lower bound, we can set
 $$
m=2,\,\, A_0 = B^{\frac{\SS}{t_1 \left(1-\SS+\frac{\SS t_0}{t_1}\right)}},\,\, A_1 = \left(\frac{B}{A_0^{t_0}}\right)^{1/t_1}.
$$
For this choice of parameters, we will have $d_0=d_1$ and so by Theorem \ref{General} we get
\begin{align*}
\hdim S_2 (  A_0,  A_1 ) &= d_0 = \inf\left\{ s\ge0 : P\left( T, -s \log | T'| - \frac{\SS^2}{t_1(1-\SS+\frac{\SS t_0}{t_1})} \log B \right)\le 0\right\} \\
& =  \inf \{s\ge 0: P(-s\log |T'|-f_{t_0,t_1 }(s)\log B)\le 0\}\\ &= \SS.
\end{align*}
This coincides with the lower bound from Theorem \ref{bhkw}.
 \subsubsection{Tan-Tian-Wang Theorem, \cite[2022]{TanTianWang}} 
  An optimal lower bound from a recent result by Tan, Tian and Wang \cite{TanTianWang} can also be extracted from our general theorem. Let us formulate their result. Consider a set
$$
E(\psi) = \{ x\in [0,1) : \exists 1 \le k\neq l \le n , a_k(x) \ge \psi(n), a_l(x) \ge \psi(n), \ \text{for infinitely many } n\in\N \}.
$$
One of the results of their paper is
\begin{theorem}
Let $\psi : \N \to \R^+$ be a non-decreasing function, and
$$
\log B = \liminf_{n\to\infty} \frac{ \log \psi(n)}{n}, \,\,\, \log b = \liminf_{n\to\infty} \frac{ \log\log \psi(n)}{n}.
$$
We see that
\begin{itemize}
\item when $1\leq B <\infty$,
$$
\hdim E(\psi) = \inf \{ s\ge 0 : P(T, -(3s-1)\log B - s \log | T'(x) | ) \leq 0 \}
$$
(remarking that $\hdim E(\psi) = 1$ if $B=1$);
\item when $B=\infty$,
$$
\hdim E(\psi) = \frac{1}{1+b}.
$$
\end{itemize}
\end{theorem}
As always, the hardest part is to prove the lower bound in the case where $B$ is finite. However, this result can be easily extracted from Remark \ref{remark1}. By the definition of $B$, we can find a subsequence $\{n_k\}_{k\geq1}$ of integers such that
$$
\log B = \lim_{k\to\infty} \frac{\psi(n_k+2)}{n_k+2} \text{\,\,\,\, and \,\,\,\,} \psi(n_k+2) \leq B^{n_k} \text { for all } k\geq 1.
$$
Next, we set $m=2, A_0= A_1 = B$ and apply the result from Remark \ref{remark1} for the sequence  $\{n_k\}_{k\geq1}$ while letting $N\to\infty$ and $M\to\infty$. This leads us to
$$
\hdim E(\psi) \geq \min \{d_0, d_1 \} = d_1,
$$
where 
\begin{align*}
d_0&= \inf \{ s\ge 0 : P(T, -s\log B - s \log | T'(x) | ) \leq 0 \}, \\
d_1 & = \inf \{ s\ge 0 : P(T, -(3s-1)\log B - s \log | T'(x) | ) \leq 0 \} .
\end{align*}
This coincides with the lower bound from Tan-Tian-Wang result.
\subsection{New results} We present one new result in this section. 

As we mentioned above, in \cite{HLS} authors have found a Hausdorff dimension for the set
\begin{equation*}
\FF_{B_1,B_2}=\left\{x\in[0,1):
\begin{split} 
 a_n(x)a_{n+1}(x) & \geq B_1^n \text{\,\, for infinitely many } n\in\N \\
 a_{n+1}(x) & < B_2^n  \text{\,\, for all sufficiently large } n\in\N
\end{split}
\right\}
\end{equation*}   

This result can be generalised to the case of a product of $m\ge2$ partial quotients.
For $m\ge2$ consider the set
\begin{equation}
\FF^m_{B_1,B_2}=\left\{x\in[0,1):
\begin{split} 
 a_{n}(x)\cdots a_{n+m-1}(x) & \geq B_1^n \text{\,\, for infinitely many } n\in\N \\
 a_{n+1}(x)\cdots a_{n+m-1}(x) & < B_2^n  \text{\,\, for all sufficiently large } n\in\N
\end{split}
\right\}.
\end{equation} 

For $t^{(m)}_{B_1}$ from \eqref{hwxconst} denote it as $t^{(m)}_{B_1}=t$ and let
$$
\theta_m = \frac{t^m-t(1-t)^{m-1}}{t^{m}-(1-t)^{m}}.
$$


\begin{theorem}\label{generalm}
For any $B_1,B_2>1$,
 \begin{itemize}
\item when $B_1^{\theta_m} \le B_2$,
$$ \hdim \FF^m_{B_1,B_2} = {t^{(m)}_{B_1}};$$
\item when  $B_1^{\theta_m} > B_2 > B_1^{1/2}$,
$$\hdim  \FF^m_{B_1,B_2} = g_{B_1,B_2};$$
\item when $B_1^{1/2}\ge B_2$,
$$ \FF^m_{B_1,B_2}=\emptyset.$$

\end{itemize}
\end{theorem}

\begin{proof}We split the proof into two cases.
\subsubsection{$B_1^{1/2}\ge B_2$}
By definition of our set we know that
$$
 a_{n+1}\cdots a_{n+m-1} < B_2^n
$$
 and
$$
 a_{n}\cdots a_{n+m-2}  < B_2^{n-1}.
$$
Multiplying two inequalities, we get 
\begin{equation}\label{simply}
 a_{n}\cdots a_{n+m-1}\leq a_{n}( a_{n+2}\cdots a_{n+m-2})^2 a_{n+m-1}  < B_2^{2n-1} \le \frac{B_1^n}{B_2} <B_1^n.
\end{equation}
This contradicts the first condition of our set, that is $a_{n}(x)\cdots a_{n+m-1}(x)  \geq B_1^n$. Hence in this case the set $ \FF^m_{B_1,B_2}$ is empty.

\subsubsection{Case $B_1^{1/2} < B_2$} First, we will work out the upper bound for cases $B_1^{\theta_m} \le B_2$ and $B_1^{\theta_m} > B_2 > B_1^{1/2}$, and then we will present a unified approach to lower bounds for both of these cases.\\
\noindent {\bf The upper bounds.} When $B_1^{\theta_m} \le B_2$ one can see that $\FF^m_{B_1,B_2} \subset E_m(B_1)$. Hence  
$$\hdim \FF^m_{B_1,B_2} \le \hdim E_m(B_1) = t_{B_1}^{(m)},$$
where $E_m(\psi)$ and $t_{B_1}^{(m)}$ were defined in \eqref{hwxconst}.
When $B_1^{\theta_m} > B_2 > B_1^{1/2}$,  consider the set
\begin{align*}
U  = \Bigg\{ x\in[0,1): 1 & \leq a_{n}(x)\cdots a_{n+m-2}(x) \leq B_2^n ,  \\
 & a_{n+m-1}(x) \geq \frac{B_1^n}{a_{n}(x)\cdots a_{n+m-2}(x)}\, \text{for infinitely many } n\in\N  \Bigg\}.
\end{align*}
Clearly, $\FF^m_{B_1,B_2} \subset U$.

The limsup nature of $U$ gives us a natural cover for it. For each $n\ge1$, define
$$
U_{n} = \left\{ x \in [0,1) : 1\le a_{n}(x)\cdots a_{n+m-2}(x) \le B_2^n , a_{n+m-1}(x) \ge   \frac{B_1^n}{a_{n}(x)\cdots a_{n+m-2}(x)} \right\}.
$$
Then $U$ can be expressed as
$$
U = \bigcap\limits_{N=1}^{\infty} \bigcup\limits_{n=N}^{\infty} U_{n}.
$$
So a cover for $U_{n}$ for each $n\ge N$ will give a cover for $U$. Naturally,
$$
U_{n} \subseteq \bigcup_{a_1,\ldots,a_{n-1}\in\N} \bigcup_{1 \leq a_{n}\cdots a_{n+m-2} \leq B_2^n} J_{n+m-2}(a_1,a_2,\ldots,a_{n+m-2}),
$$
where 
$$
J_{n+m-2}(a_1,a_2,\ldots,a_{n+m-2}) =  \bigcup_{ a_{n+m-1} \geq \frac{B_1^n}{a_{n}\cdots a_{n+m-2}}} I_{n+m-2}(a_1,a_2,\ldots,a_{n+m-2}).
$$
It is easy to see that
\begin{align*}
| J_{n+m-2}(a_1,a_2,\ldots,a_{n+m-2})  |& \asymp \frac{1}{\frac{B_1^n}{a_{n}\cdots a_{n+m-2}}q_{n+m-2}^2}\\
 & \asymp \frac{1}{B_1^n a_{n}\cdots a_{n+m-2}q^2_{n-1}}.
\end{align*}
Fix $\epsilon>0$. Then there exists $N_0 \in \N$ such that for all $n\ge N_0$, we have
$$
\frac{ (\log B_2^n)^{m-1}}{(m-1)!} \le \frac{ (\log B_1^n)^{m-1}}{(m-1)!} \le B_1^{n\epsilon}.
$$
By using this estimate with the following lemma we obtain the upper bound.
\begin{lemma}[{\cite[Lemma 4.2]{HuWuXu}}]
Let $\beta>1$. For any integer $k\ge1, 0<s<1$, we have
\begin{equation}
\sum\limits_{1\le a_{n}\cdots a_{n+m-1} \le \beta^n} \left( \frac{1}{a_n\cdots a_{n+k-1}} \right) \asymp \frac{ (\log \beta^n)^{k-1}}{(k-1)!} \beta^{n(1-s)}.
\end{equation}
\end{lemma}
Thus, the $s$-dimensional Hausdorff measure of $U$ can be estimated as
\begin{align*}
\HH^{s+\epsilon} ( U) & \le \liminf_{N\to\infty} \sum_{n\ge N} \sum_{a_1,\ldots, a_{n-1}} \sum_{1\le a_{n}\cdots a_{n+m-2} \le B_2^n } \left( \frac{1}{B_1^n a_{n}\cdots a_{n+m-2} q_{n-1}^2} \right)^{s+\epsilon} \\
&  \le \liminf_{N\to\infty} \sum_{n\ge N} \sum_{a_1,\ldots, a_{n-1}} B_2^{n(1-s)} \left( \frac{1}{B_1^n q_{n-1}^2} \right)^s.
\end{align*}
Hence,
$$
\hdim \FF^m_{B_1,B_2} \le \hdim U \leq  \inf \{s \geq 0: P(T, (1-s)\log B_2 - s\log B_1 -s \log |T'|) \leq 0 \}=g_{B_1,B_2}.
$$
\noindent {\bf The lower bounds.}
For the lower bounds, we will apply Theorem \ref{General}. Namely, let us consider a set \eqref{mainsubset} from the proof of Theorem \ref{General}:
\begin{align}
E =  \{ x\in[0,1): \,  c_i A_i^{n_k} \le a_{n_k+i}(x) < 2c_i A_i^{n_k} &\text{ for all } k\geq1, \text{ for all } 0\le i \le m-1\\
& \text{ and } a_n(x)\in\{1,\ldots,M\} \text{ for other } n\in\N \}.\notag
\end{align}
In Theorem \ref{General} we considered this set with $c_0=c_2=\cdots=c_{m-1}=1$. For this set to be a subset of  $\FF^m_{B_1,B_2}$ we need to choose suitable parameters $A_i,c_i$, where $0\leq i \leq m-1$. Denote $t = t_{B_1}^{(m)}$ and let $A_0\cdots A_{m-1} = B_1$. Now for both of our cases, let us choose the parameters $A_i$ for $i=0,\ldots,m-3$ in such a way that quantities $d_i$ for $ i=0,\ldots,m-2$ from \eqref{defd} are all equal.
This can be done by letting $A_{k} = A_{k-1}^{\frac{1-t}{t}}$, or $A_k = A_0^{(\frac{1-t}{t})^k}$ for $k=1,\ldots,m-2$. In particular, for this choice of parameters $A_i$, for $k=0,\ldots, m-2$ we have
\begin{equation}\label{uni}
\beta_k = A_0\cdots A_k =A_0^{ \frac{t^{k+1}-(1-t)^{k+1}}{t^k(2t-1)}} =  \beta_0^{ \frac{t^{k+1}-(1-t)^{k+1}}{t^k(2t-1)}}
\end{equation}
and $\beta_{m-1}=B_1$ by the choice we have previously made. By \eqref{uni} and \eqref{defd}, we have that $s(\beta_0)=d_0=\cdots=d_{m-2}$, where $s(B)$ was defined in Theorem \ref{WaWu}. Note that using \eqref{uni}, $\beta_0$ can be expressed in terms of $\beta_{m-2}$. So using definitions of $d_0=s(\beta_0)=s\left(\beta_{m-2}^\frac{t^{m-2}(2t-1)}{t^{m-1}-(1-t)^{m-1}}\right)$ and $d_{m-1}=g_{B_1,B_2}$  we see that they both depend on a free parameter $\beta_{m-2}$. Moreover, $d_0$ is decreasing and $d_{m-1}$ is increasing with respect to $\beta_{m-2}$. One can easily check that if we let
$$
\beta_{m-2} = B_1^{\theta_m},
$$
we will have $d_0=\cdots=d_{m-2}=d_{m-1}=t$, 
and so when $\beta_{m-2} < B_1^{\theta_m}$, we have $d_{m-1}< d_0$ and when $\beta_{m-2} \geq B_1^{\theta_m}$, we have $d_{m-1} \geq d_0$.
By the proof of Theorem \ref{General}, we know that for $M,N \to\infty$,
$$
\hdim E \geq \min \{ d_0, d_{m-1} \}.
$$
At this point we consider two cases.\\
{\bf Case 1:} $B_1^{\theta_m} \le B_2$. Let $\beta_{m-2} = B_1^{\theta_m}$ and
\begin{align*}
&c_0 = 2^{m-1} B_2, \\
&c_1=\frac{1}{B_2^3 2^{2(m-1)}},\\
&c_2=\cdots=c_{m-2}=1,\\
&c_{m-1} = B_2^2 2^{m-1}.
\end{align*}

We can conclude that $E$ is a subset of $\FF^m_{B_1,B_2}$,
because for our choice of $A_i,c_i$, where $0\leq i \leq m-1$, we have
\begin{align*}
&B_1^{n_k}=c_0\cdots c_{m-1}(A_0\cdots A_{m-1})^{n_k}\leq a_{n_k}\cdots a_{n_k+m-1}, \\
&B_2^{n_k-1}  >B_2^{n_k-2} \geq  \frac{1}{B_2^2} B_1^{\theta_m n_k} =2^{m-1} c_0\cdots c_{m-2} (A_0\cdots A_{m-2} )^{n_k} \geq a_{n_k}\cdots a_{n_k+m-2}, \\
& B_2^{n_k} > \frac{B_1^{\theta_m n_k}}{B_2} \geq  \frac{1}{B_2} (A_0\cdots A_{m-2} )^{n_k}  > 2^{m-1} c_1\cdots c_{m-1} (A_1\cdots A_{m-1} )^{n_k} \geq a_{n_k+1}\cdots a_{n_k+m-1},
\end{align*}  
where on the last line we used that $A_i <A_{i-1}$.

Hence by definition of our sets, set $E$ is indeed a subset of $\FF^m_{B_1,B_2}$. By the proof of Theorem \ref{General},
$$
\hdim \FF^m_{B_1,B_2} \geq  \hdim E \geq  t := t_{B_1}^{(m)} \text{ when } M,N \to\infty.
$$
Combining with the upper bound we conclude that $\hdim \FF^m_{B_1,B_2}=t_{B_1}^{(m)}$ in the case $B_1^{\theta_m} \le B_2$.\\
{\bf Case 2:} $B_1^{\theta_m} > B_2 > B_1^{1/2}$.  Let $\beta_{m-2}=B_2$. Then by what was said above, $d_{m-1}< d_0$, where $d_{m-1} =g_{B_1,B_2}  $. 
We need to make sure that $E$ is a subset of  {$\FF^m_{B_1,B_2}$.} For this we need to check that
\begin{align}
&c_0\cdots c_{m-1} (A_0\cdots A_{m-1})^{n_k} \geq B_1^{n_k}, \label{111} \\
&c_0\cdots c_{m-2} (A_0\cdots A_{m-2})^{n_k} <B_2^{n_k-1},  \label{222} \\
&c_1\cdots c_{m-1} (A_1\cdots A_{m-1} )^{n_k}<B_2^{n_k} . \label{333}
\end{align}
We can choose values for the constants $c_i$ as follows:
\begin{align*}
&c_0=c_1=\cdots=c_{m-3}=1,\\
&c_{m-2}=\frac{1}{B_2^2}, \\
&c_{m-1}=B_2^2.
\end{align*}
Now with this choice of $c_i$ inequalities \eqref{111} and \eqref{222} easily follow from $\beta_{m-1}=B_1$ and $\beta_{m-2}=B_2$. Notice that $A_{m-1} = \frac{B_1}{B_2}$.
Hence the inequality \eqref{333} is just $A_{m-1} < A_0$, which is true. (Note  that $\frac{A_k}{A_{k-1}}<1$ for all $k$.)
Now by definition of our sets, set $E$ is indeed a subset of $\FF^m_{B_1,B_2}$. By the proof of Theorem \ref{General},
$$
\hdim \FF^m_{B_1,B_2} \geq  \hdim E \geq  g_{B_1,B_2} \text{ when } M,N \to\infty.
$$
Combining with the upper bound we conclude that $\hdim \FF^m_{B_1,B_2}= g_{B_1,B_2}$ in the case $B_1^{\theta_m} > B_2 > B_1^{1/2}$.

\end{proof}



%
%
%
%
%

\end{document}